\documentclass{amsart}%
\usepackage{amsmath,amssymb,amsthm,amsfonts,euscript,textcomp,enumerate}
\usepackage{amssymb}
\usepackage{graphicx}
\usepackage{amsmath}
\usepackage{amsfonts}%
\providecommand{\U}[1]{\protect\rule{.1in}{.1in}}
\newtheorem{theorem}{Theorem}
\theoremstyle{plain}

\newtheorem{conjecture}{Conjecture}
\newtheorem{corollary}{Corollary}

\newtheorem{lemma}{Lemma}

\newtheorem{problem}{Problem}

\newtheorem{remark}{Remark}

\numberwithin{equation}{section}
\begin{document}
\title[Positive Series]{Some open problems concerning the convergence of positive series}
\author{Constantin P. Niculescu}
\address{University of Craiova, Department of Mathematics, Craiova RO-200585, Romania}
\email{cniculescu47@yahoo.com}
\author{Gabriel T. Pr\v{a}jitur\v{a}}
\address{Department of Mathematics, The College at Brockport, State University of New
York, 350 New Campus Drive, Brockport, New York 14420-2931, USA}
\email{gprajitu@brockport.edu}
\date{Jan 19, 2012}

\begin{abstract}
We discuss some old results due to Abel and Olivier concerning the convergence
of positive series and prove a set of necessary conditions involving
convergence in density.

\end{abstract}
\maketitle

\section{Introduction}

Understanding the nature of a series is usually a difficult task. The
following two striking examples can be found in Hardy's book, \emph{Orders of
infinity}: the series%
\[
\sum_{n\geq3}\frac{1}{n\ln n\left(  \ln\ln n\right)  ^{2}}%
\]
converges to 38.43..., but does it so slow that one needs to sum up its first
$10^{3.14\times10^{86}}$ terms to get the first 2 exact decimals of the sum.
In the same time, the series
\[
\sum_{n\geq3}\frac{1}{n\ln n\left(  \ln\ln n\right)  }%
\]
is divergent but its partial sums exceed 10 only after $10^{10^{100}}$ terms.
See \cite{H1910}, pp. 60-61. On page 48 of the same book, Hardy mentions an
interesting result (attributed to De Morgan and Bertrand) about the
convergence of the series of the form%
\begin{equation}
\sum_{n\geq1}\frac{1}{n^{s}}\text{ and }\sum_{n\geq n_{k}}\frac{1}{n\left(
\ln n\right)  \left(  \ln\ln n\right)  \cdots(\underset{k\text{ times}%
}{\underbrace{\ln\ln\cdots\ln n}})^{s}}~, \tag{$MB_k$}\label{serDM}%
\end{equation}
where $k$ is an arbitrarily fixed natural number, $s$ is a real number and
$n_{k}$ is a number large enough to ensure that $\underset{k\text{ times}%
}{\underbrace{\ln\ln\cdots\ln n}}$ is positive. Precisely, such a series is
convergent if $s>1$ and divergent otherwise. This is an easy consequence of
Cauchy's condensation test (see Knopp \cite{Knopp}, p. 122). Another short
argument is provided by Hardy \cite{H} in his \emph{Course of Pure
Mathematics}, on p. 376.

The above discussion makes natural the following problem.

\begin{problem}
What decides if a positive series is convergent or divergent?

Is there any universal convergence test? Is there any pattern in convergence?
\end{problem}

This is an old problem who received a great deal of attention over the years.
Important progress was made during the 19th Century by people like A.-L.
Cauchy, N. H. Abel, C. F. Gauss, A. Pringsheim and Du Bois-Reymond. In the
last fifty years the interest shifted toward combinatorial aspects of
convergence/divergence, although papers containing new tests of convergence
continue to be published. See for example \cite{A2008} and \cite{LTZ}. This
paper's purpose is to discuss the relationship between the convergence of a
positive series and the convergence properties of the summand sequence.

\section{Some history}

We start by recalling an episode from the beginning of Analysis, that marked
the moment when the series of type (\ref{serDM}) entered the attention of
mathematicians. M. Goar \cite{G1999} has written the story in more detail.

In 1827, L. Olivier \cite{Olivier1827} published a paper claiming that the
harmonic series represents a kind of ``boundary" case with which other
potentially convergent series of positive terms could be compared.
Specifically, he asserted that a positive series $\sum a_{n}$ whose terms are
monotone decreasing is convergent if and only if $na_{n}\rightarrow0.$ One
year later, Abel \cite{Abel} disproved this convergence test by considering
the case of the (divergent) positive series $\sum_{n\geq2}\frac{1}{n\ln n}$.
In the same \emph{Note}, Abel noticed two other important facts concerning the
convergence of positive series:

\begin{lemma}
\label{LemmaAbel1}There is no positive function $\varphi$ such that a positive
series $\sum a_{n}$ whose terms are monotone decreasing is convergent if and
only if $\varphi(n)a_{n}\rightarrow0.$ In other words, there is no ``boundary"
positive series.
\end{lemma}

\begin{lemma}
\label{LemmaAbel2}If $\sum a_{n}$ is a divergent positive series, then the
series $%
{\displaystyle\sum}
\left(  \frac{a_{n}}{\sum_{k=1}^{n}a_{k}}\right)  $ is also divergent. As a
consequence, for each divergent positive series there is always another one
which diverges slower.
\end{lemma}

A fact which was probably known to Abel (although it is not made explicit in
his \emph{Note}) is that the whole scale of divergent series
\begin{equation}
\sum_{n\geq n_{k}}\frac{1}{n\left(  \ln n\right)  \left(  \ln\ln n\right)
\cdots(\underset{k\text{ times}}{\underbrace{\ln\ln\cdots\ln n}})}\text{\quad
for }k=1,2,3,... \tag{A}\label{Abelscale}%
\end{equation}
comes from the harmonic series $\sum\frac{1}{n},$ by successive application of
Lemma \ref{LemmaAbel2} and the following result on the generalized Euler's constant.

\begin{lemma}
\emph{(}C. Maclaurin and A.-L. Cauchy\emph{)}. If $f$ is positive and strictly
decreasing on $[0,\infty),$ there is a constant $\gamma_{f}\in(0,f(1)]$ and a
sequence $(E_{f}(n))_{n}$ with $0<E_{f}(n)<f(n),$ such that
\begin{equation}
\sum_{k\,=\,1}^{n}\,f(k)=\int_{1}^{n}\,f(x)\,dx+\gamma_{f}+E_{f}(n)
\tag{MC}\label{euler}%
\end{equation}
for all $n.$
\end{lemma}

See \cite{Apostol}, Theorem 1, for details.

If $f(n)\rightarrow0$ as $n\rightarrow\infty,$ then $($\ref{euler}$)$ implies
\[
\sum_{k\,=\,1}^{n}\,f(k)-\int_{1}^{n}\,f(x)\,dx\rightarrow\gamma_{f}.
\]

$\gamma_{f}$ is called the \emph{generalized Euler}'\emph{s constant}, the
original corresponding to $f(x)=1/x.$

Coming back to Olivier's test of convergence, we have to mention that the
necessity part survived the scrutiny of Abel and became known as Olivier's Theorem:

\begin{theorem}
\label{ThmO}If $\sum a_{n}$ is a convergent positive series and $(a_{n})_{n}$
is monotone decreasing, then $na_{n}\rightarrow0.$
\end{theorem}

\begin{remark}
If $\sum a_{n}$ is a convergent positive series and $(na_{n})_{n}$ is monotone
decreasing, then $\left(  n\ln n\right)  a_{n}\rightarrow0$. In fact,
according to the well known estimate of harmonic numbers,%
\[
\sum_{1}^{n}\frac{1}{k}=\log n+\gamma+\frac{1}{2n}-\frac{1}{12n^{2}}%
+\frac{\varepsilon_{n}}{120n^{4}}\text{,}%
\]
where $\varepsilon_{n}\in(0,1),$ we get
\[
\sum_{\lfloor\sqrt{n}\rfloor}^{n}a_{k}=\sum_{\lfloor\sqrt{n}\rfloor}%
^{n}\left(  ka_{k}\right)  \frac{1}{k}\geq na_{n}\sum_{\lfloor\sqrt{n}\rfloor
}^{n}\frac{1}{k}\geq\frac{1}{2}\left(  n\ln n\right)  a_{n}-\frac{1}%
{2(\lfloor\sqrt{n}\rfloor-1)}%
\]
for all $n\geq2.$ Here $\lfloor x\rfloor$ denotes the largest integer that
does not exceeds $x.$
\end{remark}

Simple examples show that the monotonicity condition is vital for Olivier's
Theorem. See the case of the series $\sum a_{n},$ where $a_{n}=\frac{\log
n}{n}$ if $n$ is a square, and $a_{n}=\frac{1}{n^{2}}$ otherwise.

The next result provides an extension of the Olivier's Theorem to the context
of complex numbers.

\begin{theorem}
Suppose that $(a_{n})_{n}$ is a nonincreasing sequence of positive numbers
converging to 0 and $(z_{n})_{n}$ is a sequence of complex numbers such that
the series $\sum a_{n}z_{n}$ is convergent. Then%
\[
\lim_{n\rightarrow\infty}\left(  \sum_{k=1}^{n}z_{k}\right)  a_{n}=0.
\]

\end{theorem}

\begin{proof}
Since the series $\sum a_{n}z_{n}$ is convergent, one may choose a natural
number $m>0$ such that%
\[
\left\vert \sum_{k=m+1}^{n}a_{k}z_{k}\right\vert <\frac{\varepsilon}{4}%
\]
for every $n>m.$ We will estimate $a_{n}\left(  z_{m+1}+\cdots+z_{n}\right)  $
by using Abel's identity. In fact, letting%
\[
S_{n}=a_{m+1}z_{m+1}+\cdots+a_{n}z_{n},
\]

we get%
\begin{multline*}
\left\vert a_{n}\left(  z_{m+1}+\cdots+z_{n}\right)  \right\vert
=a_{n}\left\vert \frac{1}{a_{m+1}}a_{m+1}z_{m+1}+\cdots+\frac{1}{a_{n}}%
a_{n}z_{n}\right\vert \\
=a_{n}\left\vert \frac{1}{a_{m+1}}S_{m+1}+\frac{1}{a_{m+2}}\left(
S_{m+2}-S_{m+1}\right)  +\cdots+\frac{1}{a_{n}}\left(  S_{n}-S_{n-1}\right)
\right\vert \\
=a_{n}\left\vert \left(  \frac{1}{a_{m+1}}-\frac{1}{a_{m+2}}\right)
S_{m+1}+\cdots+\left(  \frac{1}{a_{n-1}}-\frac{1}{a_{n}}\right)  S_{n-1}%
+\frac{1}{a_{n}}S_{n}\right\vert \\
\leq\frac{\varepsilon a_{n}}{4}\left(  \left(  \frac{1}{a_{m+2}}-\frac
{1}{a_{m+1}}\right)  +\cdots+\left(  \frac{1}{a_{n}}-\frac{1}{a_{n-1}}\right)
+\frac{1}{a_{n}}\right) \\
=\frac{\varepsilon a_{n}}{4}\left(  \frac{2}{a_{n}}-\frac{1}{a_{n+1}}\right)
<\frac{\varepsilon}{2}.
\end{multline*}

Since $\lim_{n\rightarrow\infty}a_{n}=0,$ one may choose an index
$N(\varepsilon)>m$ such that%
\[
\left\vert a_{n}\left(  z_{1}+\cdots+z_{m}\right)  \right\vert <\frac
{\varepsilon}{2}%
\]
for every $n>N(\varepsilon)$ and thus%
\[
\left\vert a_{n}\left(  z_{1}+\cdots+z_{n}\right)  \right\vert \leq\left\vert
a_{n}\left(  z_{1}+\cdots+z_{m}\right)  \right\vert +\left\vert a_{n}\left(
z_{m+1}+\cdots+z_{n}\right)  \right\vert <\varepsilon
\]
for every $n>N(\varepsilon).$
\end{proof}

In 2003, T. \v{S}al\'{a}t and V. Toma \cite{ST} made the important remark that
the monotoni-\allowbreak\-city condition in Theorem \ref{ThmO} can be dropped
if the convergence of $(na_{n})_{n}$ is weakened:

\begin{theorem}
\label{ThmST}If $\sum a_{n}$ is a convergent positive series, then
$na_{n}\rightarrow0$ in density.
\end{theorem}

In order to explain the terminology, recall that a subset $A$ of $\mathbb{N}$
has \emph{zero density} if%
\[
d(A)=\underset{n\rightarrow\infty}{\lim}\frac{\left\vert A\cap\{1,\ldots
,n\}\right\vert }{n}=0,
\]
\emph{positive lower density} if%
\[
\underline{d}(A)=\underset{n\rightarrow\infty}{\lim\inf}\frac{\left\vert
A\cap\{1,\ldots,n\}\right\vert }{n}>0,
\]
and \emph{positive upper density} if%
\[
\bar{d}(A)=\underset{n\rightarrow\infty}{\lim\inf}\frac{\left\vert
A\cap\{1,\ldots,n\}\right\vert }{n}>0.
\]

Here $\left\vert \cdot\right\vert $ stands for cardinality.

We say that a sequence $(x_{n})_{n}$ of real numbers \emph{converges in
density} to a number $x$ (denoted by $(d)$-$\lim_{n\rightarrow\infty}x_{n}=x)$
if for every $\varepsilon>0$ the set $A(\varepsilon)=\left\{  n:\left\vert
x_{n}-x\right\vert \geq\varepsilon\right\}  $ has zero density. Notice that
$(d)-\lim_{n\rightarrow\infty}x_{n}=x$ if and only if there is a subset $J$ of
$\mathbb{N}$ of zero density such that%
\[
\lim_{\substack{n\rightarrow\infty\\n\notin J}}a_{n}=0.
\]
This notion can be traced back to B. O. Koopman and J. von Neumann
(\cite{KN1932}, pp. 258-259), who proved the integral counterpart of the
following result:

\begin{theorem}
\label{ThmKvN}For every sequence of nonnegative numbers,%
\[
\lim_{n\rightarrow\infty}\frac{1}{n}\sum_{k=1}^{n}a_{k}=0\Rightarrow
(d)\text{-}\lim_{n\rightarrow\infty}a_{n}=0.
\]
The converse works under additional hypotheses, for example, for bounded sequences.
\end{theorem}

\begin{proof}
Assuming $\lim_{n\rightarrow\infty}\frac{1}{n}\sum_{k=1}^{n}a_{k}=0,$ we
associate to each $\varepsilon>0$ the set $A_{\varepsilon}=\left\{
n\in\mathbb{N}:a_{n}\geq\varepsilon\right\}  .$ Since%
\begin{align*}
\frac{\left\vert \{1,...,n\}\cap A_{\varepsilon}\right\vert }{n}  &  \leq
\frac{1}{n}\sum_{k=1}^{n}\frac{a_{k}}{\varepsilon}\\
&  \leq\frac{1}{\varepsilon n}\sum_{k=1}^{n}a_{k}\rightarrow0\text{\ as
}n\rightarrow\infty,
\end{align*}
we infer that each of the sets $A_{\varepsilon}$ has zero density. Therefore
$(d)$-$\lim_{n\rightarrow\infty}a_{n}=0.$

Suppose now that $(a_{n})_{n}$ is bounded and $(d)$-$\lim_{n\rightarrow\infty
}a_{n}=0.$ Then for every $\varepsilon>0$ there is a set $J$ of zero density
outside which $a_{n}<\varepsilon.$ Since
\begin{align*}
\frac{1}{n}\sum_{k=1}^{n}a_{k}  &  =\frac{1}{n}\sum_{k\in\{1,...,n\}\cap
J}a_{k}+\frac{1}{n}\sum_{k\in\{1,...,n\}\backslash J}a_{k}\\
&  \leq\frac{\left\vert \{1,...,n\}\cap J\right\vert }{n}\cdot\sup
_{k\in\mathbb{N}}a_{k}+\varepsilon
\end{align*}
and $\lim_{n\rightarrow\infty}\frac{\left\vert \{1,...,n\}\cap J\right\vert
}{n}=0$, we conclude that $\lim_{n\rightarrow\infty}\frac{1}{n}\sum_{k=1}%
^{n}a_{k}=0.$
\end{proof}

\begin{remark}
Theorem \ref{ThmKvN} is related to the Tauberian theory, whose aim is to
provide converses to the well known fact that for any sequence of complex
numbers,%
\[
\lim_{n\rightarrow\infty}z_{n}=z\Rightarrow\lim_{n\rightarrow\infty}\frac
{1}{n}\sum_{k=1}^{n}z_{k}=z.
\]
Recall here the famous Hardy-Littlewood Tauberian theorem: If $\left\vert
z_{n}-z_{n-1}\right\vert =O\left(  1/n\right)  $ and%
\[
\lim_{n\rightarrow\infty}\frac{1}{n}\sum_{k=1}^{n}z_{k}=z,
\]
then $\lim_{n\rightarrow\infty}z_{n}=z.$ See \cite{HL1913}, Theorem 28.
\end{remark}

The aforementioned result of \v{S}al\'{a}t and Toma is actually an easy
consequence of Theorem \ref{ThmKvN}. Indeed, if $\sum a_{n}$ is a convergent
positive series, then its partial sums $S_{n}=\sum_{k=1}^{n}a_{k}$ constitutes
a convergent sequence with limit $S$. By Ces\`{a}ro's Theorem,%
\[
\lim_{n\rightarrow\infty}\frac{S_{1}+\cdots+S_{n-1}}{n}=S,
\]
whence%
\[
\lim_{n\rightarrow\infty}\frac{a_{1}+2a_{2}+\cdots+na_{n}}{n}=\lim
_{n\rightarrow\infty}\left(  S_{n}-\frac{S_{1}+\cdots+S_{n-1}}{n}\right)  =0.
\]
According to Theorem \ref{ThmKvN}, this fact is equivalent to the convergence
in density of $(na_{n})_{n}$ to 0.

In turn, the result of \v{S}al\'{a}t and Toma implies Olivier's Theorem.
Indeed, if the sequence $(a_{n})$ is decreasing, then
\[
\frac{a_{1}+2a_{2}+\dots+na_{n}}{n}\geq\frac{(1+2+\dots+n)a_{n}}{n}%
=\frac{(n+1)a_{n}}{2}%
\]
which implies that if
\[
\lim_{n\rightarrow\infty}\frac{a_{1}+2a_{2}+\cdots+na_{n}}{n}=0
\]
then $\lim_{n}na_{n}=0.$

If $\sum a_{n}$ is a convergent positive series, then so is $\sum
a_{\varphi(n)},$ whenever $\varphi:\mathbb{N\rightarrow N}$ is a bijective
map. This implies that $na_{\varphi(n)}\rightarrow0$ in density (a conclusion
that doesn't work for usual convergence).

The monograph of H. Furstenberg \cite{Fu} outlines the importance of
convergence in density in ergodic theory. In connection to series summation,
the concept of convergence in density was rediscovered (under the name of
statistical convergence) by Steinhaus \cite{St} and Fast \cite{F} (who
mentioned also the first edition of Zygmund's monograph \cite{Z}, published in
Warsaw in 1935). Apparently unaware of the Koopman-von Neumann result,
\v{S}al\'{a}t and Toma referred to these authors for the roots of convergence
in density.

At present there is a large literature about this concept and its many
applications. We only mention here the recent papers by M. Burgin and O. Duman
\cite{BD} and P. Ther\'{a}n \cite{T}.

\section{An extension of \v{S}al\'{a}t - Toma Theorem}

In this section we will turn our attention toward a generalization of the
result of \v{S}al\'{a}t and Toma mentioned above. This generalization involves
the concepts of convergence in density and convergence in lower density. A
sequence $(x_{n})_{n}$ of real numbers \emph{converges in lower density} to a
number $x$ (abbreviated, $(\underline{d})$-$\lim_{n\rightarrow\infty}x_{n}=x)$
if for every $\varepsilon>0$ the set $A(\varepsilon)=\left\{  n:\left\vert
x_{n}-x\right\vert \geq\varepsilon\right\}  $ has zero lower density.

\begin{theorem}
\label{STgen}Assume that $\sum a_{n}$ is a convergent positive series and
$(b_{n})_{n}$ is a nondecreasing sequence of positive numbers such that
$\sum_{n=1}^{\infty}\frac{1}{b_{n}}=\infty.$ Then
\[
(\underline{d})\text{-}\lim_{n\rightarrow\infty}a_{n}b_{n}=0,
\]
and this conclusion can be improved to%
\[
(d)\text{-}\lim_{n\rightarrow\infty}a_{n}b_{n}=0,
\]
provided that $\inf_{n}\frac{n}{b_{n}}>0$.
\end{theorem}

An immediate consequence is the following result about the speed of
convergence to 0 of the general term of a convergent series of positive numbers.

\begin{corollary}
If $\sum a_{n}$ is a convergent series of positive numbers, then for each
$k\in\mathbb{N},$%
\begin{equation}
(\underline{d})\,\text{-\thinspace}\lim_{n\rightarrow\infty}\left[  n\left(
\ln n\right)  \left(  \ln\ln n\right)  \cdots(\underset{k\text{ times}%
}{\underbrace{\ln\ln\cdots\ln n}})a_{n}\right]  =0. \tag{$\underline{D}%
_k$}\label{Dk}%
\end{equation}

\end{corollary}

The proof of Theorem \ref{STgen} is based on two technical lemmas:

\begin{lemma}
\label{Lemtech}Suppose that $(c_{n})_{n}$ is a nonincreasing sequence of
positive numbers such that $\sum_{n=1}^{\infty}c_{n}=\infty$ and $S$ is a set
of positive integers with positive lower density. Then the series $\sum_{n\in
S}c_{n}$ is also divergent.
\end{lemma}

\begin{proof}
By our hypothesis there are positive integers $p$ and $N$ such that%
\[
\frac{\left\vert S\cap\{1,...,n\}\right\vert }{n}>\frac{1}{p}%
\]
whenever $n\geq N$. Then $\left\vert S\cap\{1,...,kp\}\right\vert >k$ for
every $k\geq N/p,$ which yields
\begin{align*}
\sum_{n\in S}c_{n}  &  =\sum_{k=1}^{\infty}c_{n_{k}}\geq\sum_{k=1}^{\infty
}c_{kp}\\
&  =\frac{1}{p}\sum_{k=1}^{\infty}pc_{kp}\\
&  \geq\frac{1}{p}\left[  \left(  c_{p}+\cdots+c_{2p-1}\right)  +\left(
c_{2p}+\cdots+c_{3p-1}\right)  +\cdots\right] \\
&  =\frac{1}{p}\sum_{k=p}^{\infty}c_{k}=\infty.
\end{align*}

\end{proof}

Our second lemma shows that a subseries $\sum_{n\in S}\frac{1}{n}$ of the
harmonic series is divergent whenever $S$ is a set of positive integers with
positive upper density.

\begin{lemma}
\label{LemmaUppDens} If $S$ is an infinite set of positive integers and
$(a_{n})_{n\in S}$ is a nonincreasing positive sequence such that $\sum_{n\in
S}a_{n}<\infty$ and $\inf\left\{  nc_{n}:n\in S\right\}  =\alpha>0,$ then $S$
has zero density.
\end{lemma}

\begin{proof}
According to our hypotheses, the elements of $S\ $can be counted as
$k_{1}<k_{2}<k_{3}<....$ Since%
\[
0<\frac{n}{k_{n}}=\frac{na_{k_{n}}}{k_{n}a_{k_{n}}}\leq\frac{1}{\alpha
}na_{k_{n}},
\]
we infer from Theorem 1 that $\lim_{n\rightarrow\infty}\frac{n}{_{k_{n}}}=0.$
Then%
\[
\frac{\left\vert S\cap\{1,...,n)\right\vert }{n}=\frac{p}{n}=\frac{\left\vert
S\cap\{1,...,k_{p})\right\vert }{k_{p}}\leq\frac{p}{k_{p}}\rightarrow0,
\]
whence
\[
d(S)=\lim_{n\rightarrow\infty}\frac{\left\vert S\cap\{1,...,n)\right\vert }%
{n}=0.
\]

\end{proof}

\noindent\emph{Proof of Theorem} \ref{STgen}. For $\varepsilon>0$ arbitrarily
fixed we denote%
\[
S_{\varepsilon}=\left\{  n:a_{n}b_{n}\geq\varepsilon\right\}  .
\]

Then
\[
\infty>\sum\nolimits_{n\in S_{\varepsilon}}a_{n}\geq\sum\nolimits_{n\in
S_{\varepsilon}}\frac{1}{b_{n}},
\]
whence by Lemma \ref{Lemtech} it follows that $S_{\varepsilon}$ has zero lower
density. Therefore $(\underline{d})$-$\lim_{n\rightarrow\infty}a_{n}b_{n}=0$.
When $\inf_{n}\frac{n}{b_{n}}=\alpha>0,$ then%
\[
\infty>\sum\nolimits_{n\in S_{\varepsilon}}\frac{1}{b_{n}}\geq\alpha
\sum\nolimits_{n\in S_{\varepsilon}}\frac{1}{n}%
\]
so by Lemma \ref{LemmaUppDens} we infer that $S_{\varepsilon}$ has zero
density. In this case, $(d)$-$\lim_{n\rightarrow\infty}a_{n}b_{n}=0$%
.\quad$\square$

\section{Convergence associated to higher order densities}

The convergence in lower density is very weak. A better way to formulate
higher order \v{S}al\'{a}t-Toma type criteria is to consider the convergence
in harmonic density. We will illustrate this idea by proving a non-monotonic
version of Remark 1.

The harmonic density $d_{h}$ is defined by the formula%
\[
d_{h}(A)=\lim_{n\rightarrow\infty}\frac{1}{\ln n}\sum_{k=1}^{n}\frac{\chi
_{A}(k)}{k},
\]
and the limit in harmonic density, $(d_{h})$-$\lim_{n\rightarrow\infty}%
a_{n}=\ell,$ means that each of the sets $\left\{  n:\left\vert a_{n}%
-\ell\right\vert \geq\varepsilon\right\}  $ has zero harmonic density,
whenever $\varepsilon>0$. Since%
\[
d(A)=0\text{ implies }d_{h}(A)=0,
\]
(see \cite{HR}, Lemma 1, p. 241), it follows that the existence of limit in
density assures the existence of limit in harmonic density.

The harmonic density has a nice application to Benford's law, which states
that in lists of numbers from many real-life sources of data the leading digit
is distributed in a specific, non-uniform way. See \cite{Di} for more details.

\begin{theorem}
\label{ThmSTlog}If $\sum a_{n}$ is a convergent positive series, then%
\[
(d_{h})\text{-}\lim_{n\rightarrow\infty}\left(  n\ln n\right)  a_{n}=0.
\]

\end{theorem}

\begin{proof}
We start by noticing the following analogue of Lemma \ref{LemmaUppDens}: If
$(b_{n})_{n}$ is a positive sequence such that $(nb_{n})_{n}$ is decreasing
and
\[
\inf\left(  n\ln n\right)  b_{n}=\alpha>0,
\]
then every subset $S$ of $\mathbb{N}$ for which $\sum\nolimits_{n\in S}%
b_{n}<\infty$ has zero harmonic density.

To prove this assertion, it suffices to consider the case where $S$ is
infinite and to show that%
\begin{equation}
\lim_{x\rightarrow\infty}\left(  \sum\nolimits_{k\in S\cap\left\{
1,...,n\right\}  }\frac{1}{k}\right)  nb_{n}=0. \tag{$H$}\label{**}%
\end{equation}

The details are very similar to those used in Lemma \ref{LemmaUppDens}, and
thus they are omitted.

Having (\ref{**}) at hand, the proof of Theorem \ref{ThmSTlog} can be
completed by considering for each $\varepsilon>0$ the set
\[
S_{\varepsilon}=\left\{  n:\left(  n\ln n\right)  a_{n}\geq\varepsilon
\right\}  .
\]
Since
\[
\varepsilon\sum\nolimits_{n\in S_{\varepsilon}}\frac{1}{n\ln n}\leq
\sum\nolimits_{n\in S_{\varepsilon}}a_{n}<\infty,
\]
by the aforementioned analogue of Lemma \ref{LemmaUppDens} applied to
$b_{n}=1/\left(  n\ln n\right)  $ we infer that $S_{\varepsilon}$ has zero
harmonic density. Consequently ($d_{h}$)-$\lim_{x\rightarrow\infty}\left(
n\ln n\right)  a_{n}=0,$ and the proof is done.
\end{proof}

An integral version of the previous theorem can be found in \cite{NP} and
\cite{NP2}.

One might think that the fulfilment of a sequence of conditions like
$(D_{k}),$ for all $k\in\mathbb{N},$ (or something similar) using other
series, is strong enough to force the convergence of a positive series $\sum
a_{n}.$ That this is not the case was shown by Paul Du Bois-Raymond \cite{BR}
(see also \cite{Knopp}, Ch. IX, Section 41) who proved that for every sequence
of divergent positive series, each divergent essentially slower than the
previous one, it is possible to construct a series diverging slower than all
of them.

Under these circumstances the following problem seems of utmost interest:

\begin{problem}
Find an algorithm to determine whether a positive series is convergent or not.
\end{problem}

\section{The relevance of the harmonic series}

Surprisingly, the study of the nature of positive series is very close to that
of subseries of the harmonic series $\sum\frac{1}{n}.$

\begin{lemma}
\label{Lemmainteger}If $(a_{n})_{n}$ is a unbounded sequence of real numbers
belonging to $[1,\infty)$, then the series $\sum\frac{1}{a_{n}}$ and
$\sum\frac{1}{\lfloor a_{n}\rfloor}$ have the same nature.
\end{lemma}

\begin{proof}
The convergence of the series $\sum\frac{1}{a_{n}}$ follows from the
convergence of the series $\sum\frac{1}{\lfloor a_{n}\rfloor},$ by the
comparison test. In fact, $a_{n}\geq\lfloor a_{n}\rfloor,$ whence $\frac
{1}{\lfloor a_{n}\rfloor}\geq\frac{1}{a_{n}}.$

Conversely, if the series $\sum\frac{1}{a_{n}}$ is convergent, then so are the
series $\sum\frac{1}{\lfloor a_{n}\rfloor+1}$ and $\sum\frac{1}{\lfloor
a_{n}\rfloor\left(  \lfloor a_{n}\rfloor+1\right)  }.$ This is a consequence
of the inequalities%
\[
\frac{1}{\lfloor a_{n}\rfloor\left(  \lfloor a_{n}\rfloor+1\right)  }\leq
\frac{1}{\lfloor a_{n}\rfloor+1}\leq\frac{1}{a_{n}},
\]
and the comparison test. Since%
\[
\frac{1}{\lfloor a_{n}\rfloor}=\frac{1}{\lfloor a_{n}\rfloor+1}+\frac
{1}{\lfloor a_{n}\rfloor\left(  \lfloor a_{n}\rfloor+1\right)  }%
\]
we conclude that the series $\sum\frac{1}{\lfloor a_{n}\rfloor}$ is convergent too.
\end{proof}

By combining Lemma \ref{LemmaUppDens} and Lemma \ref{Lemmainteger} we infer
the following result:

\begin{corollary}
\label{Corrdens}If $(a_{n})_{n}$ is a sequence of positive numbers whose
integer parts form a set of positive upper density, then the series $\sum
\frac{1}{a_{n}}$ is divergent.
\end{corollary}

The converse of Corollary \ref{Corrdens} is not true. A counterexample is
provided by the series $\sum_{p=\text{prime}}\frac{1}{p},$ of inverses of
prime numbers, which is divergent (see \cite{A1997} or \cite{Ey1980} for a
short argument). According to an old result due to Chebyshev, if
$\pi(n)=\left\vert \left\{  p\leq n:p\text{ prime}\right\}  \right\vert ,$
then
\[
\frac{7}{8}<\frac{\pi(n)}{n/\ln n}<\frac{9}{8}%
\]
and thus the set of prime numbers has zero density.

The following estimates of the $k$th prime number,%
\[
k\left(  \ln k+\ln\ln k-1\right)  \leq p_{k}\leq k\left(  \ln k+\ln\ln
k\right)  \text{\quad for }k\geq6,
\]
which are made available by a recent paper of P. Dusart \cite{Dus}, shows that
the speed of divergence of the series $\sum_{p=\text{prime}}\frac{1}{p}$ is
comparable with that of $\sum\frac{1}{k\left(  \ln k+\ln\ln k\right)  }.$

Lemma \ref{Lemmainteger} suggests that the nature of positive series
$\sum\frac{1}{a_{n}}$ could be related to some combinatorial properties of the
sequence $(\lfloor a_{n}\rfloor)_{n}$ (of natural numbers).

\begin{problem}
Given an increasing function $\varphi:\mathbb{N\rightarrow}(0,\infty)$ with
$\lim_{n\rightarrow\infty}\varphi(n)=\infty,$ we define the upper density of
weight $\varphi$ by the formula%
\[
\bar{d}_{\varphi}(A)=\underset{n\rightarrow\infty}{\lim\sup}\frac{\left\vert
A\cap\lbrack1,n]\right\vert }{\varphi(n)}.
\]

Does every subset $A\subset\mathbb{N}$ with $\bar{d}_{n/\ln n}(A)>0$ generate
a divergent subseries $\sum_{n\in A}\frac{1}{n}$ of the harmonic series?

What about the case of other weights
\[
n/[\left(  \ln n\right)  \left(  \ln\ln n\right)  \cdots(\underset{k\text{
times}}{\underbrace{\ln\ln\cdots\ln n}})]?
\]

\end{problem}

This problem seems important in connection with the following long-standing
conjecture due to P. Erd\"{o}s:

\begin{conjecture}
$($P. Erd\"{o}s$)$. If the sum of reciprocals of a set $A$ of integers
diverges, then that set contains arbitrarily long arithmetic progressions.
\end{conjecture}

This conjecture is still open even if one only seeks a single progression of
length three. However, in the special case where the set $A$ has positive
upper density, a positive answered was provided by E. Szemer\'{e}di \cite{Sz}
in 1975. Recently, Green and T. Tao \cite{GT2004} proved Erd\"{o}s' Conjecture
in the case where $A$ is the set of prime numbers, or a relatively dense
subset thereof.

\begin{theorem}
\label{P2}Assuming the truth of Erd\"{o}s' conjecture, any unbounded sequence
$(a_{n})_{n}$ of positive numbers whose sum of reciprocals $\sum_{n}\frac
{1}{a_{n}}$ is divergent must contain arbitrarily long $\varepsilon
$-progressions, for any $\varepsilon>0$.
\end{theorem}

By an $\varepsilon$-progression of length $n$ we mean any string
$c_{1},...,c_{n}$ such that%
\[
\left\vert c_{k}-a-kr\right\vert <\varepsilon
\]
for suitable $a,r\in\mathbb{R}$ and all $k=1,...,n.$

The converse of Theorem \ref{P2} is not true. A counterexample is provided by
the convergent series $\sum_{n=1}^{\infty}\left(  \frac{1}{10^{n}+1}%
+\cdots+\frac{1}{10^{n}+n}\right)  .$

It seems to us that what is relevant in the matter of convergence is not only
the existence of some progressions but the number of them. We believe not only
that the divergent subseries of the harmonic series have progressions of
arbitrary length but that they have a huge number of such progressions and of
arbitrarily large common differences. Notice that the counterexample above
contains only progressions of common difference 1 (or subprogressions of
them). Hardy and Littlewood's famous paper \cite{HL} advanced the hypothesis
that the number of progressions of length $k$ is asymptotically of the form
$C_{k}n^{2}/\ln^{k}n$, for some constant $C_{k}$.

\end{document}